\documentclass[preprint,12pt]{article}  
\usepackage{amsmath,amssymb,amsthm}
\usepackage{mathtools}
\usepackage{microtype}
\usepackage{hyperref}  
\usepackage{graphicx}
\usepackage{booktabs}
\usepackage{enumitem}
\usepackage{cleveref}
\usepackage{url}
\usepackage{accents} 
\usepackage{setspace} 
\usepackage[a4paper, left=3.5cm, right=3.5cm]{geometry}

\newtheorem{theorem}{Theorem}

\theoremstyle{definition}

\newtheorem*{corollary*}{Corollary}
\theoremstyle{definition}
\newtheorem*{definition*}{Definition}
\newtheorem*{remark*}{Remark}

\begin{document}

\title{Asymmetry  of a class of Mellin transforms}

\author{W. Oukil\thanks{Corresponding author. Email: \texttt{oukil.walid@gmail.com}} \\
\small Faculty of Mathematics. \\
\small University of Science and Technology Houari Boumediene. \\
\small BP 32 EL ALIA 16111 Bab Ezzouar, Algiers, Algeria.}
\date{\today}

\maketitle



\begin{abstract}
We consider a class of Mellin transforms $\mu_\eta$ associated with a class of functions $\eta$, defined for every complex $s$ in the critical strip. We establish sufficient conditions on $\eta$ under which $\mu_\eta(s)$ and $\mu_\eta(1-s)$ cannot both vanish outside the critical line for sufficiently large imaginary parts of $s$. An application is given to the case in which $\eta$ is the fractional part function, for which the zeros of $\mu_\eta$ coincide with the zeros of the Riemann zeta function.
\end{abstract}

\\
\\
\begin{MSC}
44A15, 44A05, 11M06 
\end{MSC}   





\section{Introduction and Main Result} 
For a suitable function $g$, the Mellin transform is defined by
\[
\mathcal{M}\{g\}(w)=\int_0^\infty u^{-1+w}g(u)\,du,
\]
Denote by $\mathbb{C}_+$ the right half-plane defined as
\[
\mathbb{C}_+ := \Big\{ w \in \mathbb{C} : \quad \Re(w)>0\Big\}.
\]
For    $\eta\in  L^{\infty}(\mathbb{R}_+, \mathbb{R})$,  let   $\mu_\eta$ denote the function $\mu_\eta: \mathbb{C}_+\to \mathbb{C}$, defined as
\[
\mu_{\eta } (w ):=-1-(1-w) \mathcal{M}\{\eta\, \mathbf{1}_A\}(-w), \quad  \forall w\in \mathbb{C}_+, 
\]
where $\mathbf{1}_A$ is the indicator function of $A:=[1,+\infty)$.  Equivalently
\[
\mu_\eta(w )= -1-  (1-w)\int_1^{+\infty} u^{-1-w} \eta(u) \, du,\quad \forall w\in \mathbb{C}_+.
\]
The function $\mu_\eta$ is defined for all $s \in \mathbb{C}_+$. Indeed, the integral is absolutely convergent, since $\eta\in L^{\infty}(\mathbb{R}_+, \mathbb{R})$ which is bounded and since for all $w \in \mathbb{C}_+$ we have $\Re(w) > 0$.\\
We will focus on a class of functions. More precisely, we  assume that the function   $\eta\in  L^{\infty}(\mathbb{R}_+, \mathbb{R})$  satisfies both the following conditions  {\textbf{H1}} and {\textbf{H2}}:
\begin{itemize}
\item {\textbf{H1}}: There exists $\rho_\eta\in \mathbb{R}$ such that
\begin{equation*}
\sup_{t\geq1}\left|\int_1^t\,\bigl(\rho_\eta-\eta(v)\bigr)\, dv\right|<+\infty;
\end{equation*}
\item  {\textbf{H2}}: For every $r\in(0,\frac{1}{2})$ we have 
\begin{equation*}
\left| \int_1^\infty v^{-3/2}\ \bigl(\rho_\eta-\eta(v)\bigr)\,dv\right|-\left| \int_1^\infty v^{-3/2} \cosh(r\ln v) \bigl(\rho_\eta-\eta(v)\bigr)\,dv\right|>0.
\end{equation*}  
\end{itemize}~~\\
\\ 
We have the following main result:
\begin{theorem}\label{MainTheo}
Let $\eta\in  L^{\infty}(\mathbb{R}_+, \mathbb{R})$ satisfy  {\textbf{H1}} and {\textbf{H2}}.  Then for every $\sigma \in \left( 0,1\right)\setminus\{\frac{1}{2}\}$ there exists $\tau_\sigma>0$ such that for all $\tau>\tau_\sigma$ we have
\[
\left(\mu_\eta\left(s\right),\mu_\eta\left(1- s \right)\right)\neq\left(0,0\right),\qquad\text{where}\quad s:=\sigma+i\tau.
\]

\end{theorem} 
\subsection{Motivation}
The fractional part function $\eta_*(t) = \{t\}$ is bounded and locally integrable. From the integral representation of the Riemann zeta function (Titchmarsh, \cite{Titchmarsh}, page 14, Equation 2.1.5):
\begin{equation}\notag
\forall w \in \mathbb{C}_+ \setminus \{1\}: \quad \frac{1-w}{w}\zeta(w) =  \mu_{\eta_*}(w).
\end{equation}
Since the function $p(t) = \int_1^t (1/2 - \{u\})\,du$ is $1$-periodic and vanishes at integer values, it is bounded for all $t \ge 1$.  More precisely,  by   Dirichlet's Theorem  we have
\begin{equation}\notag
\small p(t)=\sum_{j\in \mathbb{Z}^*} \frac{1-  \exp(i j2\pi t)}{(j2\pi)^2}>0,\quad  \forall u\in (1,+\infty)\setminus\mathbb{N},  
\end{equation}
Consequently, the  fractional part function satisfies   hypotheses {\textbf{H1}}   with $\rho_{\eta_*} = \frac{1}{2}$. The integration by part gives
 \[
\int_1^\infty v^{-1-x} \bigl(\rho_{\eta_*}-\eta_*(v)\bigr)\,dv>0,\qquad \forall x\in(0,1),
\]
Denote $\phi(r):(0,\frac12)\to\mathbb{R}$ as
\begin{align*}
\phi(r):= \int_1^\infty v^{-3/2}\bigl(1-\cosh(r\ln v)\bigr)\bigl(\rho_{\eta_*}-\eta_*(v)\bigr)\,dv.
\end{align*}
Again, using the integration by part formula,  we get $\phi'(r)>0$ for all $r\in (0, \frac12)$. In other words it is strictly increasing  on $(0,\frac12)$. Since $\phi(0)=0$,  hence the function  $\phi$ is positive  on $(0,\frac12)$. By consequence, the hypothesis   {\textbf{H2}}  is satisfied.

\section{Proof of the Theorem \ref{MainTheo}} \label{ProofTheorem}

\begin{proof}[Proof of the Theorem \ref{MainTheo}]
Let $s \in\mathbb{C}$ such that  $\Re(s)\in\left( 0,1\right)\setminus\{\frac{1}{2}\}$. Assume that 
\[\Big(\mu_\eta(s),\,  \mu_\eta(1-{s})\Big)= (0,\, 0).\]
In order to simplify the notation, denote
\[
\eta^\sharp(t) :=  t^{-\frac{1}{2}} \Big(\rho_\eta-\eta(t)\Big),\quad \forall t\geq1,
\]
and for a suitable function $g$,    define the linear operator  $\Psi$  as
\[
\Psi[g](t):=\int_t^\infty v^{-1} g(v)\, dv.
\]
Use the integration by parts formula and hypothesis  {\textbf{H1}}, we get
\begin{equation}\label{omega}
\sup_{t\geq1}\left|t^{\frac{3}{2}}\,  \Psi[ \eta^\sharp](t ) \right|<+\infty.
\end{equation}
Since $\Re(s)\in (0,1)\setminus\{\frac{1}{2}\}$, then $ |\Re\big(s-\frac{1}{2}\big)|<\frac{1}{2}$.  Since $\big(\mu_\eta(s),\,  \mu_\eta(1-{s})\big)= (0,\, 0)$;   Using the previous equation, the integration by formula gives
\begin{equation}\label{enfin}
\frac{2\, \rho_\eta}{s(1-s)} =2\Psi[ \eta^\sharp](1)+2 z \,I(s)+\frac{1}{z} \,J(s), 
\end{equation}
where 
\begin{align*}
I(s)&:=  \int_1^\infty\, t^{-1}\,  \sinh\big(z \, \ln(t)\big)\, \Psi[ \eta^\sharp](t)\, dt,\qquad z=s-\frac{1}{2},\\
J(s)&:=  \int_1^\infty\, t^{-1}\,  \sinh\big(z \, \ln(t)\big)\,   \eta^\sharp (t)\, dtt.
\end{align*} 
Denote $\sigma+i\tau:=s$ with $\sigma$ and $\tau$ real numbers. Denote $r:=\sigma-\frac{1}{2}$. Since $\eta^\sharp$ is a real valued function then same for $\Psi[ \eta^\sharp]$.
\[
\Im\big( I(s)\big) = \int_1^\infty\,    t^{-1  }\, \sin\big(\tau \ln(t)\big) \cosh\big(r \ln(t)\big)  \Psi[ \eta^\sharp](t)\, dt.
\]
Use the change of variable $t\to e^{u}:=\exp(u)$, we get
\[
\Im\big( I(s)\big) = \int_0^\infty\,  \sin(\tau u) \,\cosh(r \, u) \Psi[ \eta^\sharp](e^u)\, du.
\]
For  $\tau\neq0$, we can  write 
\begin{align*}
  \Im\Big(I(s)\Big) &= \int_{ 0}^{ \frac{2 \pi}{\tau}}\,   \sin(\tau u)  \sum_{n=0}^{+\infty} \, \cosh\Big(r \,  (u+\frac{2n\, \pi}{ \tau})\Big)  \Psi[ \eta^\sharp]\big(e^{u+\frac{2n\, \pi}{ \tau}}\big) du.
\end{align*} 
Use the change of variable $u\to \frac{2\pi}{ \tau} v$, we obtain
\begin{equation}\label{IdentIm}
 \Im\Big(I(s)\Big) = \frac{2\pi}{ \tau} \int_{ 0}^{ 1}\,   \sin(2\pi\, v)  \sum_{n=0}^{+\infty} \, \cosh\Big(r \, \frac{2\pi}{ \tau} (v+n)\Big)   \Psi[ \eta^\sharp]\big(e^{\frac{2\pi}{ \tau} (v+n)}\big) dv.
\end{equation} 
By \eqref{omega}, for every fixed $v\in [0,1]$, we have
\[
\lim_{x\to+\infty} \cosh\big(r \, \frac{2\pi}{ \tau} (v+x)\big) \Psi[ \eta^\sharp]\big(e^{\frac{   2\pi}{ \tau}(v+x) }=0, 
\]
Since
\[
\frac{d}{dt}\Psi[g](e^t)=-g(e^t),
\]
and for $t>0$ we have
\begin{align*}
\sum_{k=0}^{n} \, \cosh\Big(r \, t(v+k)\Big)&=e^{ r\, t \, v}\frac{1- e^{ r\,(n+1) t}  }{1-e^{ r\,t}}+e^{-r\, v\, t} \frac{1- e^{ -r\,(n+1) t}   }{1-e^{-r\,t}}\\
&=\frac{e^{ r\, t \, v}  }{1-e^{ r\,t}}+\frac{e^{-r\, v\, t}   }{1-e^{-r\,t}}\\
&- \frac{e^{r\, t\, \big( v+ (n+1)  \big)}  }{1-e^{ r\,t}}- \frac{ e^{ -r\, t\, \big( v+ (n+1)  \big)}   }{1-e^{-r\,t}},\quad n\in \mathbb{N},
\end{align*}
by Abel summation formula, we have the following equality,
\begin{align*}
\sum_{n=0}^{+\infty} \, \cosh\Big(r \, \frac{2\pi}{ \tau} (v+n)\Big)   \Psi[ \eta^\sharp]\big(e^{\frac{2\pi}{ \tau} (v+n)}\big) =  \frac{\Psi[ \eta^\sharp](1)}{2} \, g_r(\frac{ 2  \pi}{ \tau},   v)-\frac{1}{2}\, h_r(\frac{ 2  \pi}{ \tau},   v), 
\end{align*}
where
\begin{align*}
g_r(t,  v)&:= t\left(\frac{e^{ r\, t \, v}  }{1-e^{ r\,t}}+\frac{e^{-r\, v\, t}   }{1-e^{-r\,t}} \right) \,\\
h_r(t,  v)&:= t\int_0^{+\infty} \left( \frac{e^{r\, t\, \big( v+ (\lfloor  x \rfloor+1)  \big)}  }{1-e^{ r\,t}}+ \frac{ e^{ -r\, t\, \big( v+ (\lfloor  x \rfloor+1)  \big)}   }{1-e^{-r\,t}}\right) \eta^\sharp\big(e^{t(v+x)}\big)\,dx.
\end{align*}
Equation \eqref{IdentIm} implies
\begin{equation}\notag
 \Im\Big(I(s)\Big) = \frac{ \pi}{ \tau} \, \Psi[ \eta^\sharp](1)\, \int_{ 0}^{ 1}\,   \sin(2\pi\, v) g_r(\frac{ 2  \pi}{ \tau},  v) dv+\frac{ \pi}{ \tau} \int_{ 0}^{ 1}\,   \sin(2\pi\, v)  h_r(\frac{ 2  \pi}{ \tau},  v) dv.
\end{equation} 
This equation holds for $s:=\sigma+i\tau$ satisfying
\[\Big(\mu_\eta(s),\,  \mu_\eta(1-{s})\Big)= (0,\, 0).\]
Proof of the theorem by contradiction. Fixe $\sigma \in\left( 0,1\right)\setminus\{\frac{1}{2}\}$  and suppose that there exists a sequence $(\tau_n)_{n\geq 1}$  such that  $\tau_n\to+\infty$ and
\[\Big(\mu_\eta(s_n),\,  \mu_\eta(1-{s}_n)\Big)= (0,\, 0),\quad s_n:=\sigma+i\tau_n,\quad n\geq1.\]
Denote $t_n:=\frac{ 2  \pi}{ \tau_n}$. Since
\[
\lim_{t_n\to0^+}    \left| \int_{ 0}^{1}\,   \sin( 2\pi\,v )  g_r(t_n,  v) dv\right|=0,
\]
 then
\begin{align}
\notag\lim_{ n\to+\infty  }   \left|\tau_n  \Im\left(I(s_n)\right)  \right|&\le  \pi \,\lim_{t_n\to0^+}    \left| \int_{ 0}^{1}\,   \sin( 2\pi\,v )  h_r(t_n,  v) dv\right| \\
\notag&\le  \ \lim_{t_n\to0^+}   \max_{v\in[0,\frac{1}{2}]} \left|h_r(t_n,  v)-  h_r(t_n,  v+\frac{1}{2})\right|,
\end{align}  
Use the change of varaible $x t_n\to\nu$, we obtain
\[
h_r(t_n,  v)= \int_0^\infty \,  \, f_r(t_n,   v,\nu)\eta^\sharp(e^{ t_n\, v+\nu})\,d\nu,
\] 
where
\[
f_r(t_n,  v, \nu):= \frac{ e^{r\, t_n\left( v+ (\lfloor\frac{\nu}{t_n}\rfloor+1)\right)}}{1-e^{r\, t_n}} +  \frac{ e^{-r\, t_n\,\left( v+(\lfloor\frac{\nu}{t_n}\rfloor+1)\right)}}{1-e^{-r\, t_n  }}, \quad n\in \mathbb{N}.
\]
We use the fact that $\lfloor x\rfloor=x-\{x\}$, which gives $\lim_ {n\to+\infty}t_n\, \lfloor\frac{\nu}{t_n}\rfloor= \nu$. Hence
\[
\lim_{t_n\to0^+}  \left[f_r(t_n,  v, \nu)-   f_r(t_n,  v+\frac{1}{2}, \nu) \right]= \cosh(r\,\nu),
\] 
unifotrmly for $v\in[0,\frac{1}{2}]$. We recall that
\[
\eta^\sharp(t) :=  t^{-\frac{1}{2}} \Big(\rho_\eta-\eta(t)\Big),\quad \forall t\geq1,
\]
and $\eta\in L^\infty([1,+\infty, \mathbb{R})$. By Lebesgue Theorem, we get
\begin{align}
\label{limit}\lim_{ n \to +\infty}   \left|\tau_n  \Im\left(I(s_n)\right)  \right| \le \left|  \int_0^\infty e^{-\frac{1}{2} \nu} \cosh\big(r \, \nu\big) \Big(\rho_\eta-\eta(e^\nu)\Big)\, d\nu \right|.
\end{align}
We recall that,
\begin{align*}
J(s)=   \int_1^\infty\, t^{-1}\,  \sinh\big(z \, \ln(t)\big)\,   \eta^\sharp (t)\, dt
\end{align*} 
Since $|\rho_\eta-\eta(t)|\le |\rho_\eta|+\|\eta\|_\infty\le 2\|\eta\|_\infty$ and since $z:=\sigma-\frac{1}{2}+i\tau$, then
\[
|J(s)|\le \frac{2}{\sigma(1-\sigma)}\|\eta\|_\infty.
\] 
 Equation \eqref{enfin} implies
\begin{align*}\notag
  \left|   \Im(z_n^{-1}) \right| \Psi[ \eta^\sharp](1)-&\left|   \Im\left(I(s_n)\right)\right|\\
&\le \frac{1}{\sigma(1-\sigma)}\|\eta\|_\infty |\Im(z_n^{-2})|+\left|\Im\left(\frac{ \rho_\eta}{z_n\,s_n(1-s_n)}\right)\right|,
\end{align*} 
where $z_n:=s_n-\frac{1}{2}$. Multiply by $\tau_n$;  as $n \to+\infty$   equation \eqref{limit} gives
\[
\left|   \Psi[ \eta^\sharp](1)\right| -  \left|  \int_0^\infty e^{-\frac{1}{2} u} \cosh\big(r \, u\big) \Big(\rho_\eta-\eta(e^u)\Big)\, du\right|\le0.
\]
We obtain a contradiction with hypothesis {\textbf{H2}}.
\end{proof}  

\makeatother


\end{document}